\documentclass[11pt]{article}
\usepackage{amssymb,amsfonts,amsmath,latexsym,epsf,tikz,url}
\usepackage[usenames,dvipsnames]{pstricks}
\usepackage{pstricks-add}
\usepackage{epsfig}
\usepackage{pst-grad} 
\usepackage{pst-plot} 
\usepackage[space]{grffile} 
\usepackage{etoolbox} 
\makeatletter 
\patchcmd\Gread@eps{\@inputcheck#1 }{\@inputcheck"#1"\relax}{}{}
\makeatother

\newtheorem{theorem}{Theorem}[section]

\newtheorem{conjecture}[theorem]{Conjecture}
\newtheorem{corollary}[theorem]{Corollary}

\newtheorem{remark}[theorem]{Remark}
\newtheorem{example}[theorem]{Example}
\newtheorem{definition}[theorem]{Definition}

\newcommand{\qed}{\hfill $\square$\medskip}

\begin{document}

\title{A graph related to Euler $\phi$ function}

\author{
Nima Ghanbari
\and
Saeid Alikhani$^{}$\footnote{Corresponding author}
}

\date{\today}

\maketitle

\begin{center}
Department of Informatics, University of Bergen, P.O. Box 7803, 5020 Bergen, Norway\\
Department of Mathematics, Yazd University, 89195-741, Yazd, Iran\\
{\tt Nima.ghanbari@uib.no$^{}$, alikhani@yazd.ac.ir}
\end{center}


\begin{abstract}
Euler function $\phi(n)$ is the number of positive integers less than $n$ and relatively prime to $n$.  
 Suppose that  $\phi^1(n)=\phi(n)$ and $\phi^i(n)=\phi(\phi^{i-1}(n))$.  
Let $A\subseteq \mathbb{N}$,  and
	 $A_{\phi}=\{ \phi^k(n)|  n\in A , k\in \mathbb{N} \cup \{0\}\}.$
	 We consider a graph $G_{\phi}(A)=(V,E)$, where $V=A_{\phi}$ and $E=\{\{r,s\}| r,s\in V, \phi(r)=s \}$. We say a graph $H$ is a $G_\phi$-graph, if  there exists a set of natural numbers $A$, such that $H=G_{\phi}(A)$. 
In this paper we study the graph $G_{\phi}(A)$ and investigate some specific graphs and some chemical trees as $G_\phi$-graph. 
\end{abstract}


\medskip
\noindent{\bf AMS Subj.\ Class.:} 05C05, 11AXX.

\section{Introduction}

One of the topics in  number theory is the Euler $\phi$ function, which is denoted by  $\phi(n)$ and is the number of positive integers less than $n$ and relatively prime
to $n$.  Gauss's theorem state that the sum of $\phi(d)$ over the
divisors $d$ of $n$ is $n$. In other words, 
$$\sum_{d|n}\phi(d) = n.$$

Two facts about the Euler $\phi$ function use for evaluating $\phi(n)$. First, if $p$ is prime, then
$\phi(p^k)=p^k-p^{k-1}$ and second, $\phi$ is multiplicative; that is, if $m$ and $n$ are relatively prime, then $\phi(mn)=\phi(m)\phi(n)$.
The number $15$ has interesting properties that  
$$15=\phi(15)+\phi(\phi(15))+\phi(\phi(\phi(15)))+\phi(\phi(\phi(\phi(15)))).$$ 
 Loomis, Plytage and Polhill in \cite{sum} asked for which numbers does this happen? 
 Let $\phi^0(n)=n$, $\phi^1(n)=\phi(n)$ and $\phi^i(n)=\phi(\phi^{i-1}(n))$. So can iterate $\phi$ to create
 the sequence $\{n,\phi(n), \phi^2(n),...\}$. Parts of this sequence can be show as a graph  (see \cite{sum}). 
 Following Pillai \cite{Pillai}, let $R(n)$ denote the smallest integer $k$ such that $\phi^k(n)=1$, in other words, $R(n)$
  is the number of steps it takes the sequence beginning with $n$ to reach $1$. 
     Loomis, Plytage and Polhill made the following definitions:  
 
 Define $\Phi(n)$ by $\Phi(n)=\sum_{i=1}^{R(n)} \phi^i(n)$. They called $n$ a perfect totient number (PTN) if $\Phi(n)=n$.
 They proved that a prime power $p^k$ is a PTN if and only if $p=3$. Also they showed that if $n$ is a PTN and $4n+1$ is prime, then $3(4n+1)$ is also a PTN.
 
  A graph $G$ is a pair $G=(V,E)$, where $V$ and $E$ are the vertex set and the edge set of $G$, respectively.  The degree or valency of a vertex $u$ in a graph $G$ (loopless),
 denoted by $deg(u)$, is the number of edges meeting at $u$. The distance between two vertices $u$ and $v$ of $G$, denoted by $d(u,v)$, is
 defined as the least length of the  walks between them. As usual we denote the path and cycle of order $n$ by $P_n$ and $C_n$, respectively. Also
 $K_{1,n}$ is the star graph with $n+1$ vertices. 

 \medskip
 In the next section, we introduce a new graph related to Euler $\phi$ function and  investigate the properties of this graph.Also, we consider some specific graphs and some chemical trees as $G_\phi$-graph, in Section 3.

\section{A graph related to Euler $\phi$ function}

In this section, we state the definition of a new graph which is related to Euler $\phi$ function and investigate its properties.

	\begin{definition}
	For any set of natural numbers $A$, let
	 $A_{\phi}=\{ \phi^k(n)|  n\in A , k\in \mathbb{N} \cup \{0\}\}.$
	 Graph related to Euler $\phi$ function is denoted by $G_{\phi}(A)$ and is a graph with vertex set $V=A_{\phi}$ and edge set $E=\{\{r,s\}| r,s\in V, \phi(r)=s \}$. 
	 We say a graph $H$ is a $G_\phi$-graph, if  there exists a set of natural numbers $A$, such that $H=G_{\phi}(A)$.
	\end{definition}

	\begin{example}
	Let $A=\{ 3, 7 , 11 , 20\}\subseteq \mathbb{N}$. We have
	\begin{center}
	$\phi(20)=8, ~ \phi(8)=\phi^2(20)=4,~\phi(4)=\phi^3(20)=2,~\phi(2)=\phi^4(20)=1,$\\
	$\phi(11)=10,~ \phi(10)=\phi^2(11)=4,~\phi(4)=\phi^3(11)=2,~\phi(2)=\phi^4(11)=1,$\\
	$\phi(7)=6,~ \phi(6)=\phi^2(7)=2,~\phi(2)=\phi^3(7)=1$ and \\
	$\phi(3)=2, ~\phi(2)=\phi^2(3)=1.$
	\end{center}
	So we have $A_\phi=\{ 1, 2, 3, 4, 6, 7,8,10 , 11 , 20\}$ and therefore the graph $G_{\phi}(A)$ is a graph  in Figure \ref{1}.
	\label{example}
	\end{example}

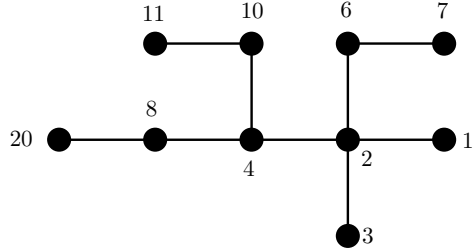
\begin{figure}
\begin{center}
\psscalebox{0.8 0.8}
{
\begin{pspicture}(0,-7.1635575)(7.74,-3.0564423)
\psdots[linecolor=black, dotsize=0.4](0.82,-5.366442)
\psdots[linecolor=black, dotsize=0.4](2.42,-5.366442)
\psdots[linecolor=black, dotsize=0.4](4.02,-5.366442)
\psdots[linecolor=black, dotsize=0.4](5.62,-5.366442)
\psdots[linecolor=black, dotsize=0.4](7.22,-5.366442)
\psdots[linecolor=black, dotsize=0.4](4.02,-3.7664423)
\psdots[linecolor=black, dotsize=0.4](5.62,-3.7664423)
\psdots[linecolor=black, dotsize=0.4](2.42,-3.7664423)
\psdots[linecolor=black, dotsize=0.4](5.62,-6.966442)
\psdots[linecolor=black, dotsize=0.4](7.22,-3.7664423)
\psline[linecolor=black, linewidth=0.04](0.82,-5.366442)(7.22,-5.366442)(7.22,-5.366442)
\psline[linecolor=black, linewidth=0.04](5.62,-3.7664423)(5.62,-6.966442)(5.62,-6.966442)
\psline[linecolor=black, linewidth=0.04](5.62,-3.7664423)(7.22,-3.7664423)(7.22,-3.7664423)
\psline[linecolor=black, linewidth=0.04](2.42,-3.7664423)(4.02,-3.7664423)(4.02,-5.366442)(4.02,-5.366442)
\rput[bl](0.0,-5.466442){20}
\rput[bl](2.26,-4.966442){8}
\rput[bl](3.88,-5.966442){4}
\rput[bl](5.84,-5.8064423){2}
\rput[bl](7.52,-5.506442){1}
\rput[bl](5.5,-3.3264422){6}
\rput[bl](7.1,-3.3464422){7}
\rput[bl](3.84,-3.3464422){10}
\rput[bl](2.2,-3.3864422){11}
\rput[bl](5.86,-7.0864425){3}
\end{pspicture}
}
\end{center}
\caption{Graph related to Euler $\phi$ function of Example \ref{example}.} \label{1}
\end{figure}

	\begin{theorem}      
		\begin{enumerate} 
			\item[(i)] 
			For any set of natural numbers $A$, the graph $G_{\phi}(A)$ is a tree.
			\item[(ii)] 
			The distance between two vertices $i$ ($i\in G_{\phi}(A)$) and $1$ is $d(i,1)=R(i)$.
		\end{enumerate}
	\end{theorem} 
	
	\begin{proof}
	\begin{enumerate}
	\item[(i)]
	Suppose that $G_{\phi}(A)$ is not a tree. Therefore, there is a cycle of length at least 3. Without loss of generality, suppose that the length of cycle is 3. So there are vertices  $a,b$ and $c$ and the cycle is $abca$. By the definition of $G_{\phi}(A)$, we suppose that $\phi(a)=b$, $\phi(b)=c$ and $\phi(c)=a$. Therefore 
	\begin{center}
	$\phi(a)=b$, $\phi^2(a)=\phi(b)=c$ and $\phi^3(a)=\phi(c)=a$,
	\end{center}
	which is a contradiction with the fact that for every $n>1$, and $k\neq 1$, $\phi^n(k)>k$.
	\item[(ii)]
	It follows from the definitions. \qed
	\end{enumerate}
	\end{proof}
	
	Here we ask a main question: Which trees can be a $G_{\phi}$-graph? In the following theorem, we show that the path graph $P_n$ can be a  $G_{\phi}$-graph. 
	
	\begin{theorem} \label{thm1}
		The path graph $P_n$ ($n>1$) is a $G_{\phi}$ graph. 
	\end{theorem}
	    
    \begin{proof}
	Since $\phi^{n-1}  (2^{n-1})=\phi ^{n-2} (2^{n-2})=\phi^{n-3}(2^{n-3})= \ldots =1$, it suffices to consider $A=\{ 2^{n-1}\}$. So  
	$A_\phi=\{ 1, 2,4,\ldots, 2^{n-2},2^{n-1}\}$ and $G_{\phi}(A)= P_n$.
	\qed
	\end{proof}

    The following result is an immediate result of Theorem \ref{thm1}:

	\begin{corollary} 
	For every natural number $m$, there exist a set of natural numbers $A$, such that  $G_{\phi}(A)$ has size $m$.
	\end{corollary} 

Also we have the following theorem: 
 
	\begin{theorem} 
		For every natural number $n$, there exist a set of natural numbers $A$, such that $|A|=n$ and  $G_{\phi}(A)$ has size $n$.
	\end{theorem}
	
	\begin{proof}
	It suffices to consider
	$A=\{ 1, 2,4,\ldots, 2^{n-1}\}$. The rest is similar to  the proof of Theorem \ref{thm1}.
	\qed
	\end{proof}

	\begin{theorem} \label{thm2}
		Let $G$ be the $G_{\phi}$-graph related to a set $A$ of size $n$. If  $t$ is the number of its leaves, then $1\leq t\leq n+1$. Also there is a set $B$ such that the number of  leaves of $G_{\phi}(B)$ is $t$ and $|B|=n$ .
	\end{theorem} 
	
	\begin{proof}
	Suppose that $A$ is a set of size $n$. In the worst case, the set $A$ consists of  odd prime numbers. Suppose that $p$ is a prime number. So $\phi(p)=p-1$ which is an even number and  so there is no edge between any numbers in $A$. Therefore $\{p,p-1\}$ are  leaves   of the graph. Since 1 is in $G_{\phi}(A)$, so the graph has at most $n+1$ leaves. When $A=\{2\}$, then the graph has only one leaf. To prove the second part, consider the odd prime numbers $p_1,p_2,\ldots,p_{t-1}$ with 1 in $B$ and add $n-t$ members of the set 
	$$\{\phi(p_1),\phi^2(p_1),\phi^3(p_1),\ldots,\phi(p_2),\phi^2(p_2),\phi^3(p_2),\ldots,\phi(p_{t-1}),\phi^2(p_{t-1}),\phi^3(p_{t-1}),\ldots,2\}$$ to $B$, and for being confident about the size of this set, it is suffices to consider enough large  prime numbers.
	\qed
	\end{proof}

\medskip 	

Here we state the following example (regarding to the proof of the Theorem \ref{thm2}). 
	
	\begin{example}\label{ex2}
    Let $n=5$ and $t=3$. Consider a set $A$ containing three odd prime number  $A=\{1,7,11\}$ and add two numbers $2$ and $4$ from $\{\phi(11)=10,\phi^2(11)=4,\phi(7)=6,\phi^2(7)=2\}$ to it. So $B=\{1,2,4,7,11\}$. As we see in Figure \ref{2}, $G_{\phi}(B)$ has $3$ leaves with $|B|=5$ as desired.
    
    \begin{figure}
    \begin{center}
    \psscalebox{0.8 0.8}
      {
\begin{pspicture}(0,-5.383558)(8.394231,-2.9564424)
\psdots[linecolor=black, dotsize=0.4](0.19711548,-3.5864422)
\psdots[linecolor=black, dotsize=0.4](1.7971154,-3.5864422)
\psdots[linecolor=black, dotsize=0.4](3.3971155,-3.5864422)
\psdots[linecolor=black, dotsize=0.4](4.9971156,-3.5864422)
\psdots[linecolor=black, dotsize=0.4](6.5971155,-3.5864422)
\psdots[linecolor=black, dotsize=0.4](8.197116,-3.5864422)
\psdots[linecolor=black, dotsize=0.4](4.9971156,-5.1864424)
\psline[linecolor=black, linewidth=0.04](0.19711548,-3.5864422)(8.197116,-3.5864422)(4.9971156,-3.5864422)(4.9971156,-5.1864424)(4.9971156,-5.1864424)
\rput[bl](0.05711548,-3.2864423){11}
\rput[bl](1.6571155,-3.2264423){10}
\rput[bl](3.2771156,-3.2664423){4}
\rput[bl](4.9371157,-3.2464423){2}
\rput[bl](6.4571157,-3.2264423){6}
\rput[bl](8.077115,-3.2664423){7}
\rput[bl](5.3971157,-5.3264422){1}
\end{pspicture}
}
      \end{center}
      \caption{Graph related to  Example \ref{ex2}.} \label{2}
      \end{figure}
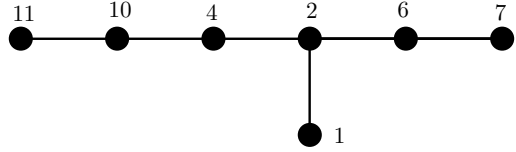

	\end{example}

	\begin{theorem} \label{star}
		The star graph $K_{1,n}$ cannot be a $G_{\phi}$ graph, for $n>4$.
	\end{theorem}
	
	\begin{proof}
	First we show that the equation $\phi(n)=2$ has only three solutions. Suppose that $n=p_1^{t_1}p_2^{t_2}\ldots p_k^{t_k}$. Then $\phi(n)=2$ means that $$2=p_1^{t_1-1}(p_1-1)p_2^{t_2-1}(p_2-1)\ldots p_k^{t_k-1}(p_k-1).$$
	The only primes $p_i$ in the factorization of $n$ must be such that $p_i-1$ divides $2$. So the only primes possible in $n$ are $2$ and  $3$ and so we have 
	$n=2^{a}3^{b}$ where $a,b\geq 0$ and $2=2^{a-1}(1)3^{b-1}(2)$. If $b>1$ then $3|2$ which is a contradiction. So we have the following cases:
        \begin{enumerate} 
			\item[(i)] 
			If $b=1$, then  $n=2^{a}3$ and $\phi(n)= 2^{a-1}(2-1)3^{1-1}(2)$. If $a>1$ then $\phi(n)= 2^{a}\neq 2$. If $a=1$ then $\phi(n)=(2-1)(3-1)=2$ and we have a solution which is $n=6$. If $a=0$ then $\phi(n)=(3-1)=2$ and we have another solution which is $n=3$.
			\item[(ii)] 
			If $b=0$, then we have $n=2^{a}$ and $\phi(n)= 2^{a-1}=2$ which implies $a=2$. So we have another solution which is $n=4$.
		\end{enumerate}
	Therefore there is no $n\neq 3,4,6$ such that $\phi(n)=2$ and we have just three solutions. Now we try to construct a star graph $K_{1,n}$. One of the leaves should be $1$. By the same argument we did before, it is easy to see that there is only one number which satisfies the equation $\phi(n)=1$ and that is $2$. So the neighbour of $1$ is $2$. Since there are only three solution for the equation $\phi(n)=2$, we can not have more than $4$ adjacent vertices with $2$ and therefore there is no star graph for $n>4$.
	\qed
	\end{proof}

A banana tree, $B(n,m)$ is a graph obtained by connecting one leaf of each $n$ copies of an $m$-star graph to a new vertex. By Theorem \ref{star} we have the following result:

	\begin{corollary}
		The banana tree $B(n,m)$ cannot be a $G_{\phi}$-graph, for $m\geq 6$.
	\end{corollary}

    \begin{remark}\label{remark1}
    The star graph $K_{1,n}$ is a  $G_{\phi}$ graph, for $n<4$. It suffices to consider  $\{1,2\}$ for $n=1$, $\{1,2,3\}$ for $n=2$ , $\{1,2,3,4\}$ for $n=3$ and $\{1,2,3,4,6\}$ for $n=4$.
    \end{remark}

    \medskip 
    The corona of two graphs $G_{1}$ and $G_{2}$ is the graph $G = G_{1}\circ G_{2}$ formed from one copy of $G_{1}$ and $\vert V(G_{1}) \vert$ copies of $G_{2}$, where the $i$th vertex of $G_{1}$ is adjacent to every vertex in the $i$th copy of $G_{2}$. The corona $G \circ K_{1}$, in particular, is the graph constructed from a copy of $G$, where for each vertex $v\in V(G)$, a new vertex $v'$ and a pendant edge $vv'$ are added ~\cite{A}.
    
  	\begin{theorem} \label{thmcent}
		The centipede graph $P_n\circ K_1$ is a $G_{\phi}$ graph, for $n>1$.
	\end{theorem}
	
	\begin{proof}
    It suffices to consider  $A=\{1,2,4,\ldots,2^{n},12,24,48,\ldots,3(2^{n})\}$. So $G_{\phi}(A)$ is $P_n\circ K_1$ as shown in Figure \ref{cent}. Therefore $P_n\circ K_1$ is a $G_{\phi}$ graph.
	\qed
	\end{proof}

	\begin{figure}
	\begin{center}
\psscalebox{0.8 0.8} 
{
\begin{pspicture}(0,-5.565)(10.947116,-2.415)
\psdots[linecolor=black, dotsize=0.4](0.19711548,-3.225)
\psdots[linecolor=black, dotsize=0.4](1.7971154,-3.225)
\psdots[linecolor=black, dotsize=0.4](3.3971155,-3.225)
\psdots[linecolor=black, dotsize=0.4](4.9971156,-3.225)
\psdots[linecolor=black, dotsize=0.4](6.5971155,-3.225)
\psdots[linecolor=black, dotsize=0.4](10.5971155,-3.225)
\psdots[linecolor=black, dotsize=0.4](10.5971155,-4.825)
\psdots[linecolor=black, dotsize=0.4](6.5971155,-4.825)
\psdots[linecolor=black, dotsize=0.4](4.9971156,-4.825)
\psdots[linecolor=black, dotsize=0.4](3.3971155,-4.825)
\psdots[linecolor=black, dotsize=0.4](1.7971154,-4.825)
\psdots[linecolor=black, dotsize=0.4](0.19711548,-4.825)
\psline[linecolor=black, linewidth=0.04](6.5971155,-3.225)(6.5971155,-4.825)
\psline[linecolor=black, linewidth=0.04](7.3971157,-3.225)(0.19711548,-3.225)(7.7971153,-3.225)(7.7971153,-3.225)(7.3971157,-3.225)
\psline[linecolor=black, linewidth=0.04](10.197116,-3.225)(9.397116,-3.225)(10.5971155,-3.225)(10.5971155,-4.825)(10.5971155,-4.825)
\psline[linecolor=black, linewidth=0.04](4.9971156,-3.225)(4.9971156,-4.825)(4.9971156,-4.825)
\psline[linecolor=black, linewidth=0.04](3.3971155,-3.225)(3.3971155,-4.825)(3.3971155,-4.825)
\psline[linecolor=black, linewidth=0.04](1.7971154,-3.225)(1.7971154,-4.825)(1.7971154,-4.825)
\psline[linecolor=black, linewidth=0.04](0.19711548,-3.225)(0.19711548,-4.825)(0.19711548,-4.825)
\psdots[linecolor=black, dotsize=0.1](8.197116,-3.225)
\psdots[linecolor=black, dotsize=0.1](8.5971155,-3.225)
\psdots[linecolor=black, dotsize=0.1](8.997115,-3.225)
\rput[bl](0.037115477,-5.325){1}
\rput[bl](0.05711548,-2.885){2}
\rput[bl](1.6771154,-2.905){4}
\rput[bl](3.2371154,-2.905){8}
\rput[bl](4.7571154,-2.925){16}
\rput[bl](6.4371157,-2.885){32}
\rput[bl](10.497115,-2.885){2}
\rput[bl](10.657116,-2.605){n}
\rput[bl](10.157116,-5.565){3(2   )}
\rput[bl](1.5571154,-5.385){12}
\rput[bl](3.1571155,-5.405){24}
\rput[bl](4.7771153,-5.385){48}
\rput[bl](6.4771156,-5.385){96}
\rput[bl](10.617115,-5.205){n}
\rput[bl](10.157116,-5.565){3(2     )}
\end{pspicture}
}
\end{center}
\caption{Graph related to  Theorem \ref{thmcent}} \label{cent}
\end{figure}
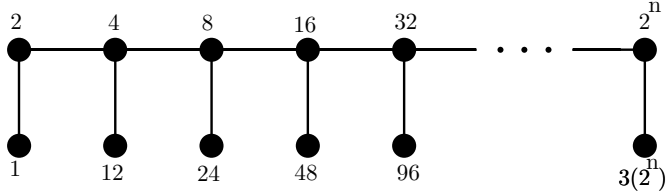

\begin{figure}
	\begin{center}
		\psscalebox{0.5 0.5} 
		{
			\begin{pspicture}(0,-6.5014424)(12.794231,1.2956733)
			\psdots[linecolor=black, dotsize=0.4](10.197116,-3.3014421)
			\psdots[linecolor=black, dotsize=0.4](12.5971155,-2.1014423)
			\psdots[linecolor=black, dotsize=0.4](11.397116,-0.50144225)
			\psdots[linecolor=black, dotsize=0.4](12.5971155,-4.5014424)
			\psdots[linecolor=black, dotsize=0.4](11.397116,-6.1014423)
			\psdots[linecolor=black, dotsize=0.4](7.7971153,-3.3014421)
			\psline[linecolor=black, linewidth=0.08](7.7971153,-3.3014421)(10.197116,-3.3014421)(11.397116,-0.50144225)(11.397116,-0.50144225)
			\psline[linecolor=black, linewidth=0.08](10.197116,-3.3014421)(12.5971155,-2.1014423)(12.5971155,-2.1014423)
			\psline[linecolor=black, linewidth=0.08](10.197116,-3.3014421)(12.5971155,-4.5014424)(12.5971155,-4.5014424)
			\psline[linecolor=black, linewidth=0.08](10.197116,-3.3014421)(11.397116,-6.1014423)(11.397116,-6.1014423)
			\psellipse[linecolor=blue, linewidth=0.04, linestyle=dotted, dotsep=0.10583334cm, dimen=outer](6.5971155,-3.1014423)(5.2,3.4)
			\psdots[linecolor=black, dotsize=0.1](6.5971155,-3.3014421)
			\psdots[linecolor=black, dotsize=0.1](6.1971154,-3.3014421)
			\psdots[linecolor=black, dotsize=0.1](5.7971153,-3.3014421)
			\pscustom[linecolor=black, linewidth=0.04]
			{
				\newpath
				\moveto(3.7971156,-1.7014422)
			}
			\pscustom[linecolor=black, linewidth=0.04]
			{
				\newpath
				\moveto(5.7971153,-1.3014423)
			}
			\pscustom[linecolor=black, linewidth=0.04]
			{
				\newpath
				\moveto(5.7971153,-0.9014423)
			}
			\pscustom[linecolor=black, linewidth=0.04]
			{
				\newpath
				\moveto(5.7971153,-1.3014423)
			}
			\pscustom[linecolor=black, linewidth=0.04]
			{
				\newpath
				\moveto(-0.60288453,2.2985578)
			}
			\rput[bl](11.397116,-0.10144226){1}
			\rput[bl](9.797115,-2.9014423){2}
			\psdots[linecolor=black, dotsize=0.4](8.197116,1.0985577)
			\psdots[linecolor=black, dotsize=0.4](7.3971157,1.0985577)
			\psdots[linecolor=black, dotsize=0.4](6.5971155,1.0985577)
			\psdots[linecolor=black, dotsize=0.4](5.7971153,1.0985577)
			\psdots[linecolor=black, dotsize=0.4](0.59711546,-1.3014423)
			\psdots[linecolor=black, dotsize=0.4](0.19711548,-2.1014423)
			\psdots[linecolor=black, dotsize=0.4](0.19711548,-2.9014423)
			\psdots[linecolor=black, dotsize=0.4](0.59711546,-3.7014422)
			\rput[bl](4.1971154,-2.1014423){{\huge H}}
			\psdots[linecolor=black, dotsize=0.1](1.7971154,-5.301442)
			\psdots[linecolor=black, dotsize=0.1](2.1971154,-5.7014422)
			\psdots[linecolor=black, dotsize=0.1](2.5971155,-6.1014423)
			\psline[linecolor=black, linewidth=0.04](8.197116,1.0985577)(7.3971157,-0.9014423)(7.3971157,-0.9014423)
			\psline[linecolor=black, linewidth=0.04](7.3971157,1.0985577)(6.9971156,-0.9014423)(6.9971156,-0.9014423)
			\psline[linecolor=black, linewidth=0.04](6.5971155,1.0985577)(6.5971155,-0.9014423)(6.5971155,-0.9014423)
			\psline[linecolor=black, linewidth=0.04](5.7971153,1.0985577)(6.1971154,-0.9014423)(6.1971154,-0.9014423)
			\psline[linecolor=black, linewidth=0.04](0.59711546,-1.3014423)(2.1971154,-2.1014423)(2.1971154,-2.1014423)
			\psline[linecolor=black, linewidth=0.04](0.19711548,-2.1014423)(2.1971154,-2.5014422)(2.1971154,-2.5014422)
			\psline[linecolor=black, linewidth=0.04](0.19711548,-2.9014423)(2.1971154,-2.9014423)(2.1971154,-2.9014423)
			\psline[linecolor=black, linewidth=0.04](0.59711546,-3.7014422)(2.1971154,-3.3014421)(2.1971154,-3.3014421)
			\end{pspicture}
		}
	\end{center}
	\caption{Graph $ H\circ \overline{K_4}$ related to the proof of  Theorem \ref{hok}} \label{hok4}
\end{figure}
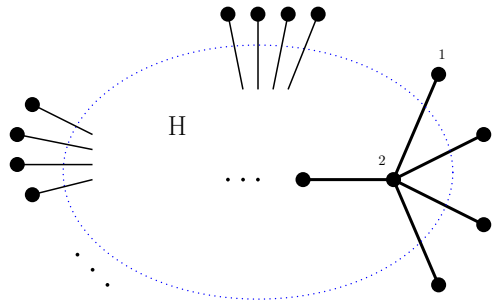

    \begin{theorem} \label{hok}
		If $H$ is a $G_\phi$-graph, then $ H\circ \overline{K_n}$ is not a $G_{\phi}$-graph, for $n>3$.
	\end{theorem}
	
	\begin{proof}
    Suppose that $H$ is a $G_\phi$-graph and without loss of generality let $n=4$. We consider graph $ H\circ \overline{K_4}$. If $ H\circ \overline{K_4}$ is a  $G_\phi$-graph, then one of its leaves should be $1$ and its neighbour should be $2$ (see Figure \ref{hok4}). Then the vertex $2$ has at least five  neighbours which one of them is $1$. On the other hand the equation $\phi(x)=2$ has only three solutions which are $3,4$ and $6$. So by assigning these three numbers to the three leaves adjacent to vertex $2$, at least  one vertex left without assigning any number. So there is no set of natural numbers such that  $ H\circ \overline{K_4}$ be a $G_{\phi}$-graph. By the same argument we conclude that $ H\circ \overline{K_n}$ is not a $G_{\phi}$-graph, for $n>4$.
	\qed
	\end{proof}

   \begin{theorem}\label{leaves}
		If  $H$ is  a $G_\phi$-graph with $m$ leaves and $G_{\phi}(A)=H$, for a set $A$, then  $|A|\geq m-1$.
	\end{theorem}
	
	\begin{proof}
    By the definition, except the leaf $1$, we need all the leaves  in the set $A$. So we have $|A|\geq m-1$.
	\qed
	\end{proof}

    \begin{remark}
    The lower bound in the Theorem \ref{leaves} is sharp. It suffices to consider examples in Remark \ref{remark1}.
    \end{remark}

\section{Some specific chemical trees as $G_{\phi}$-graph} 
    
     In chemical graphs, the vertices of the graph correspond to the atoms of the molecule, and the edges represent the chemical bonds. In this section, we consider some well known molecules and investigate them as $G_{\phi}$-graph. We start with the following easy theorem:

    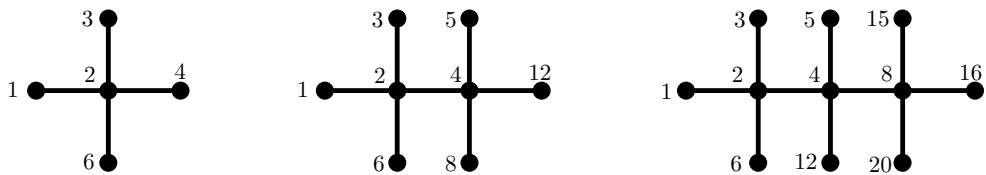
\begin{figure}
    	\begin{center}
    		\psscalebox{0.8 0.8} 
    		{
    			\begin{pspicture}(0,-5.205)(16.227837,-2.495)
    			\psdots[linecolor=black, dotsize=0.3](0.48,-3.845)
    			\psdots[linecolor=black, dotsize=0.3](1.68,-3.845)
    			\psdots[linecolor=black, dotsize=0.3](2.88,-3.845)
    			\psdots[linecolor=black, dotsize=0.3](1.68,-2.645)
    			\psdots[linecolor=black, dotsize=0.3](1.68,-5.045)
    			\psdots[linecolor=black, dotsize=0.3](5.28,-3.845)
    			\psdots[linecolor=black, dotsize=0.3](6.48,-3.845)
    			\psdots[linecolor=black, dotsize=0.3](7.68,-3.845)
    			\psdots[linecolor=black, dotsize=0.3](8.88,-3.845)
    			\psdots[linecolor=black, dotsize=0.3](6.48,-2.645)
    			\psdots[linecolor=black, dotsize=0.3](7.68,-2.645)
    			\psdots[linecolor=black, dotsize=0.3](6.48,-5.045)
    			\psdots[linecolor=black, dotsize=0.3](7.68,-5.045)
    			\psdots[linecolor=black, dotsize=0.3](11.28,-3.845)
    			\psdots[linecolor=black, dotsize=0.3](12.48,-3.845)
    			\psdots[linecolor=black, dotsize=0.3](13.68,-3.845)
    			\psdots[linecolor=black, dotsize=0.3](14.88,-3.845)
    			\psdots[linecolor=black, dotsize=0.3](16.08,-3.845)
    			\psdots[linecolor=black, dotsize=0.3](12.48,-5.045)
    			\psdots[linecolor=black, dotsize=0.3](13.68,-5.045)
    			\psdots[linecolor=black, dotsize=0.3](14.88,-5.045)
    			\psdots[linecolor=black, dotsize=0.3](12.48,-2.645)
    			\psdots[linecolor=black, dotsize=0.3](13.68,-2.645)
    			\psdots[linecolor=black, dotsize=0.3](14.88,-2.645)
    			\psline[linecolor=black, linewidth=0.08](2.88,-3.845)(2.88,-3.845)(0.48,-3.845)(0.48,-3.845)
    			\psline[linecolor=black, linewidth=0.08](5.28,-3.845)(8.88,-3.845)(8.88,-3.845)
    			\psline[linecolor=black, linewidth=0.08](11.28,-3.845)(16.08,-3.845)(16.08,-3.845)
    			\psline[linecolor=black, linewidth=0.08](1.68,-2.645)(1.68,-5.045)(1.68,-5.045)
    			\psline[linecolor=black, linewidth=0.08](6.48,-2.645)(6.48,-5.045)(6.48,-5.045)
    			\psline[linecolor=black, linewidth=0.08](7.68,-2.645)(7.68,-5.045)(7.68,-5.045)
    			\psline[linecolor=black, linewidth=0.08](12.48,-2.645)(12.48,-5.045)(12.48,-5.045)
    			\psline[linecolor=black, linewidth=0.08](13.68,-2.645)(13.68,-5.045)(13.68,-5.045)
    			\psline[linecolor=black, linewidth=0.08](14.88,-2.645)(14.88,-5.045)(14.88,-5.045)
    			\rput[bl](0.0,-3.945){1}
    			\rput[bl](1.28,-3.705){2}
    			\rput[bl](1.24,-2.765){3}
    			\rput[bl](2.78,-3.645){4}
    			\rput[bl](1.26,-5.145){6}
    			\rput[bl](4.82,-3.945){1}
    			\rput[bl](6.1,-3.725){2}
    			\rput[bl](6.06,-2.785){3}
    			\rput[bl](7.36,-3.705){4}
    			\rput[bl](7.28,-2.785){5}
    			\rput[bl](6.08,-5.185){6}
    			\rput[bl](7.28,-5.185){8}
    			\rput[bl](8.66,-3.665){12}
    			\rput[bl](10.86,-3.985){1}
    			\rput[bl](12.04,-3.705){2}
    			\rput[bl](12.08,-2.785){3}
    			\rput[bl](13.32,-3.705){4}
    			\rput[bl](13.24,-2.785){5}
    			\rput[bl](12.02,-5.185){6}
    			\rput[bl](14.52,-3.685){8}
    			\rput[bl](13.08,-5.145){12}
    			\rput[bl](15.82,-3.665){16}
    			\rput[bl](14.32,-5.205){20}
    			\rput[bl](14.28,-2.765){15}
    			\end{pspicture}
    		}
    	\end{center}
    	\caption{Methane, Ethane and Propane respectively.} \label{chem1}
    \end{figure}

     \begin{theorem}
     	\begin{enumerate} 
     		\item[(i)] 
     		Methane, Ethane and Propane are $G_{\phi}$-graph. 
     		
     		\item[(ii)] Butane and Isobutane are $G_\phi$-graph.
     		
     		\item[(iii)] Pentane and Isopentane are $G_\phi$ graphs.
     		\end{enumerate} 
     	\end{theorem} 
          \begin{proof} 
    \begin{enumerate} 
     			\item[(i)] 
    Methane is a $G_\phi$-graph with $A_\phi=\{ 3,4,6\}$, Ethane is a $G_\phi$-graph with $A_\phi=\{ 3,5,6,8,12\}$ and Propane is a $G_\phi$-graph with $A_\phi=\{ 3,5,6,12,15,16,20\}$ (see Figure \ref{chem1}).
    
    \item[(ii)] Butane is a $G_\phi$-graph with $A_\phi=\{ 3,5,6,12,15,16,20\}$ and Isobutane is a $G_\phi$-graph with $A_\phi=\{ 3,5,6,13,15,20,21,24,28\}$ as wee see in Figure \ref{chem2}.

    \item[(iii)]	Pentane is a $G_\phi$-graph with $A_\phi=\{ 3,5,6,12,15,17,20,48,64,80,96,\}$ and Isopentane is a $G_\phi$-graph with $A_\phi=\{3,5,6,12,13,15,20,28,32,40,48\}$ as wee see in Figure \ref{chem3}.\qed
\end{enumerate}
    \end{proof}

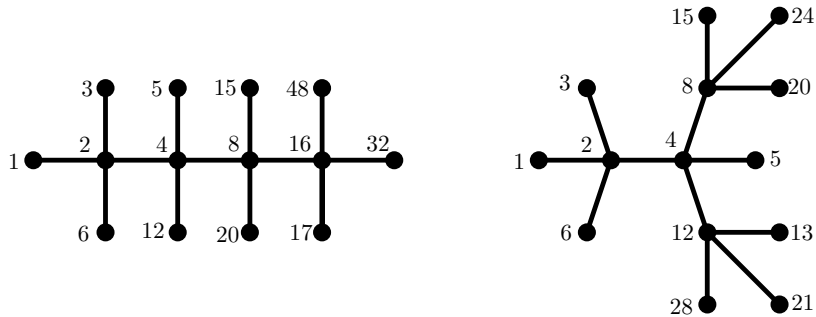
\begin{figure}
\begin{center}
\psscalebox{0.8 0.8} 
{
\begin{pspicture}(0,-5.205)(13.41,-0.095)
\psdots[linecolor=black, dotsize=0.3](0.42,-2.645)
\psdots[linecolor=black, dotsize=0.3](1.62,-2.645)
\psdots[linecolor=black, dotsize=0.3](2.82,-2.645)
\psdots[linecolor=black, dotsize=0.3](4.02,-2.645)
\psdots[linecolor=black, dotsize=0.3](5.22,-2.645)
\psdots[linecolor=black, dotsize=0.3](1.62,-3.845)
\psdots[linecolor=black, dotsize=0.3](2.82,-3.845)
\psdots[linecolor=black, dotsize=0.3](4.02,-3.845)
\psdots[linecolor=black, dotsize=0.3](1.62,-1.445)
\psdots[linecolor=black, dotsize=0.3](2.82,-1.445)
\psdots[linecolor=black, dotsize=0.3](4.02,-1.445)
\psline[linecolor=black, linewidth=0.08](0.42,-2.645)(5.22,-2.645)(5.22,-2.645)
\psline[linecolor=black, linewidth=0.08](1.62,-1.445)(1.62,-3.845)(1.62,-3.845)
\psline[linecolor=black, linewidth=0.08](2.82,-1.445)(2.82,-3.845)(2.82,-3.845)
\psline[linecolor=black, linewidth=0.08](4.02,-1.445)(4.02,-3.845)(4.02,-3.845)
\rput[bl](0.0,-2.785){1}
\rput[bl](1.18,-2.505){2}
\rput[bl](1.22,-1.585){3}
\rput[bl](2.46,-2.505){4}
\rput[bl](2.38,-1.585){5}
\rput[bl](1.16,-3.985){6}
\rput[bl](3.66,-2.485){8}
\rput[bl](2.22,-3.945){12}
\rput[bl](4.66,-2.485){16}
\rput[bl](3.46,-4.005){20}
\rput[bl](3.42,-1.565){15}
\psdots[linecolor=black, dotsize=0.3](8.82,-2.645)
\psdots[linecolor=black, dotsize=0.3](10.02,-2.645)
\psdots[linecolor=black, dotsize=0.3](11.22,-2.645)
\psdots[linecolor=black, dotsize=0.3](12.42,-2.645)
\psdots[linecolor=black, dotsize=0.3](9.62,-3.845)
\psdots[linecolor=black, dotsize=0.3](11.62,-1.445)
\psdots[linecolor=black, dotsize=0.3](12.82,-1.445)
\psdots[linecolor=black, dotsize=0.3](12.82,-0.245)
\psdots[linecolor=black, dotsize=0.3](11.62,-0.245)
\psdots[linecolor=black, dotsize=0.3](11.62,-3.845)
\psdots[linecolor=black, dotsize=0.3](12.82,-3.845)
\psdots[linecolor=black, dotsize=0.3](11.62,-5.045)
\psdots[linecolor=black, dotsize=0.3](12.82,-5.045)
\psline[linecolor=black, linewidth=0.08](8.82,-2.645)(12.42,-2.645)(12.42,-2.645)
\psline[linecolor=black, linewidth=0.08](11.22,-2.645)(11.62,-1.445)(11.62,-0.245)(11.62,-0.245)
\psline[linecolor=black, linewidth=0.08](11.62,-1.445)(12.82,-0.245)(12.82,-0.245)
\psline[linecolor=black, linewidth=0.08](11.62,-1.445)(12.82,-1.445)(12.82,-1.445)
\psline[linecolor=black, linewidth=0.08](11.22,-2.645)(11.62,-3.845)(11.62,-3.845)
\psline[linecolor=black, linewidth=0.08](11.62,-3.845)(12.82,-3.845)(12.82,-3.845)
\psline[linecolor=black, linewidth=0.08](11.62,-3.845)(12.82,-5.045)(12.82,-5.045)
\psline[linecolor=black, linewidth=0.08](11.62,-3.845)(11.62,-5.045)(11.62,-5.045)
\psdots[linecolor=black, dotsize=0.3](9.62,-1.445)
\psline[linecolor=black, linewidth=0.08](9.62,-1.445)(10.02,-2.645)(9.62,-3.845)(9.62,-3.845)
\rput[bl](8.4,-2.785){1}
\rput[bl](9.52,-2.505){2}
\rput[bl](9.16,-1.485){3}
\rput[bl](10.92,-2.425){4}
\rput[bl](12.66,-2.745){5}
\rput[bl](9.18,-3.965){6}
\rput[bl](11.2,-1.525){8}
\rput[bl](11.02,-0.385){15}
\rput[bl](13.02,-0.365){24}
\rput[bl](12.96,-1.545){20}
\rput[bl](13.0,-3.965){13}
\rput[bl](13.02,-5.145){21}
\rput[bl](11.0,-5.205){28}
\rput[bl](11.02,-3.965){12}
\psdots[linecolor=black, dotsize=0.3](6.42,-2.645)
\psdots[linecolor=black, dotsize=0.3](5.22,-1.445)
\psdots[linecolor=black, dotsize=0.3](5.22,-3.845)
\psline[linecolor=black, linewidth=0.08](5.22,-1.445)(5.22,-3.845)(5.22,-2.645)(6.42,-2.645)(6.42,-2.645)
\rput[bl](4.62,-1.585){48}
\rput[bl](5.96,-2.485){32}
\rput[bl](4.68,-3.965){17}
\end{pspicture}
}
\end{center}
\caption{Butane and Isobutane respectively.} \label{chem2}
\end{figure}
   
\begin{figure}
	\begin{center}
		\psscalebox{0.7 0.7} 
{
\begin{pspicture}(0,-4.246082)(15.367837,0.5417549)
\psdots[linecolor=black, dotsize=0.3](0.42,-2.8060818)
\psdots[linecolor=black, dotsize=0.3](1.62,-2.8060818)
\psdots[linecolor=black, dotsize=0.3](2.82,-2.8060818)
\psdots[linecolor=black, dotsize=0.3](4.02,-2.8060818)
\psdots[linecolor=black, dotsize=0.3](5.22,-2.8060818)
\psdots[linecolor=black, dotsize=0.3](1.62,-4.0060816)
\psdots[linecolor=black, dotsize=0.3](2.82,-4.0060816)
\psdots[linecolor=black, dotsize=0.3](4.02,-4.0060816)
\psdots[linecolor=black, dotsize=0.3](1.62,-1.6060817)
\psdots[linecolor=black, dotsize=0.3](2.82,-1.6060817)
\psdots[linecolor=black, dotsize=0.3](4.02,-1.6060817)
\psline[linecolor=black, linewidth=0.08](0.42,-2.8060818)(5.22,-2.8060818)(5.22,-2.8060818)
\psline[linecolor=black, linewidth=0.08](1.62,-1.6060817)(1.62,-4.0060816)(1.62,-4.0060816)
\psline[linecolor=black, linewidth=0.08](2.82,-1.6060817)(2.82,-4.0060816)(2.82,-4.0060816)
\psline[linecolor=black, linewidth=0.08](4.02,-1.6060817)(4.02,-4.0060816)(4.02,-4.0060816)
\rput[bl](0.0,-2.9460816){1}
\rput[bl](1.18,-2.6660817){2}
\rput[bl](1.22,-1.7460817){3}
\rput[bl](2.46,-2.6660817){4}
\rput[bl](2.38,-1.7460817){5}
\rput[bl](1.16,-4.146082){6}
\rput[bl](3.66,-2.6460817){8}
\rput[bl](2.22,-4.1060815){12}
\rput[bl](4.66,-2.6460817){16}
\rput[bl](3.46,-4.166082){20}
\rput[bl](3.42,-1.7260817){15}
\psdots[linecolor=black, dotsize=0.3](6.42,-2.8060818)
\psdots[linecolor=black, dotsize=0.3](5.22,-1.6060817)
\psdots[linecolor=black, dotsize=0.3](5.22,-4.0060816)
\psline[linecolor=black, linewidth=0.08](5.22,-1.6060817)(5.22,-4.0060816)(5.22,-2.8060818)(6.42,-2.8060818)(6.42,-2.8060818)
\rput[bl](4.62,-4.146082){48}
\rput[bl](5.92,-2.6460817){32}
\psdots[linecolor=black, dotsize=0.3](7.62,-2.8060818)
\psdots[linecolor=black, dotsize=0.3](6.42,-4.0060816)
\psdots[linecolor=black, dotsize=0.3](6.42,-1.6060817)
\psline[linecolor=black, linewidth=0.08](6.42,-1.6060817)(6.42,-4.0060816)(6.42,-4.0060816)
\rput[bl](7.12,-2.6660817){64}
\rput[bl](5.88,-4.1260815){96}
\rput[bl](5.86,-1.7260817){80}
\psline[linecolor=black, linewidth=0.08](6.42,-2.8060818)(7.62,-2.8060818)(7.62,-2.8060818)
\psdots[linecolor=black, dotsize=0.3](9.22,-2.8060818)
\psdots[linecolor=black, dotsize=0.3](10.42,-2.8060818)
\psdots[linecolor=black, dotsize=0.3](11.62,-2.8060818)
\psdots[linecolor=black, dotsize=0.3](12.82,-2.8060818)
\psdots[linecolor=black, dotsize=0.3](14.02,-2.8060818)
\psdots[linecolor=black, dotsize=0.3](10.42,-4.0060816)
\psdots[linecolor=black, dotsize=0.3](11.62,-4.0060816)
\psdots[linecolor=black, dotsize=0.3](12.82,-4.0060816)
\psdots[linecolor=black, dotsize=0.3](10.42,-1.6060817)
\psdots[linecolor=black, dotsize=0.3](11.62,-0.8060817)
\psdots[linecolor=black, dotsize=0.3](12.82,-1.6060817)
\psline[linecolor=black, linewidth=0.08](9.22,-2.8060818)(14.02,-2.8060818)(14.02,-2.8060818)
\psline[linecolor=black, linewidth=0.08](10.42,-1.6060817)(10.42,-4.0060816)(10.42,-4.0060816)
\psline[linecolor=black, linewidth=0.08](11.62,-1.6060817)(11.62,-4.0060816)(11.62,-4.0060816)
\psline[linecolor=black, linewidth=0.08](12.82,-1.6060817)(12.82,-4.0060816)(12.82,-4.0060816)
\rput[bl](8.8,-2.9460816){1}
\rput[bl](9.98,-2.6660817){2}
\rput[bl](10.02,-1.7460817){3}
\rput[bl](11.26,-2.6660817){4}
\rput[bl](9.96,-4.146082){6}
\rput[bl](12.46,-2.6460817){8}
\rput[bl](11.1,-0.6460817){12}
\rput[bl](13.46,-2.6460817){16}
\rput[bl](12.26,-4.166082){20}
\rput[bl](12.22,-1.7260817){15}
\psdots[linecolor=black, dotsize=0.3](15.22,-2.8060818)
\psdots[linecolor=black, dotsize=0.3](14.02,-1.6060817)
\psdots[linecolor=black, dotsize=0.3](14.02,-4.0060816)
\psline[linecolor=black, linewidth=0.08](14.02,-1.6060817)(14.02,-4.0060816)(14.02,-2.8060818)(15.22,-2.8060818)(15.22,-2.8060818)
\rput[bl](13.42,-4.146082){48}
\rput[bl](14.72,-2.6460817){32}
\rput[bl](4.62,-1.7060817){17}
\psdots[linecolor=black, dotsize=0.3](12.82,-0.8060817)
\psdots[linecolor=black, dotsize=0.3](10.42,-0.8060817)
\psdots[linecolor=black, dotsize=0.3](11.62,0.3939183)
\psline[linecolor=black, linewidth=0.08](11.62,0.3939183)(11.62,-2.0060816)(11.62,-2.0060816)
\psline[linecolor=black, linewidth=0.08](10.42,-0.8060817)(12.82,-0.8060817)(12.82,-0.8060817)
\rput[bl](11.22,-4.146082){5}
\rput[bl](11.04,0.23391831){13}
\rput[bl](9.84,-0.8860817){21}
\rput[bl](13.06,-0.9060817){28}
\rput[bl](13.54,-1.7060817){40}
\end{pspicture}
}
	\end{center}
	\caption{Pentane and Isopentane respectively.} \label{chem3}
\end{figure}
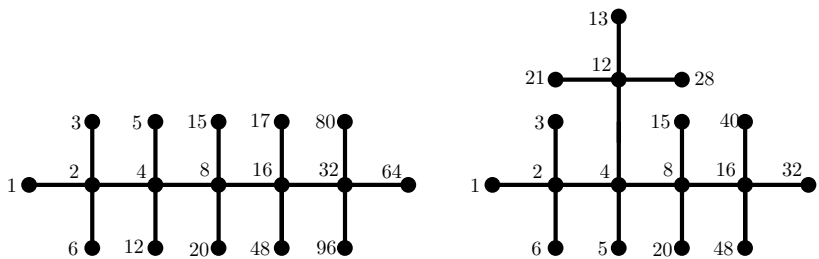

    \begin{theorem} \label{neop}
		Neopentane is not a $G_\phi$-graph.
	\end{theorem}
	
	\begin{proof}
    Consider  Neopentane as see in Figure \ref{neop1}. The vertex labeled with $1$ should be one of the leaves  and the vertex labeled with $2$ should be its adjacent vertex. Now we consider  the vertex $u$. As  see in the proof of Theorem \ref{star}, the vertex $u$ should be $3$,  $4$ or $6$. By the similar argument in the proof of   Theorem \ref{star}, we conclude that the equation $\phi(n)=3$ has no solution. The numbers  $7,9,14,18$ are the solutions of $\phi(n)=6 $ and $5,8,10,12$ are the solutions of $\phi(n)=4$. So $u$ is not $3$. Let $u=6$. So we should choose three numbers of $7,9,14,18$ to give to $v,w$ and $x$. each of them has 3 neighbours. But $\phi(n)=14$ and $\phi(n)=7$ have no solutions. Thus $6$ is not suitable for $u$. By the same argument we conclude that $4$ is not suitable for $u$, too, since  $\phi(n)=10$ has only two solutions and $\phi(n)=5$ has no solutions. Therefore we can not give the vertex $u$ any number. Hence, Neopentane is not a $G_\phi$-graph.
	\qed
	\end{proof}

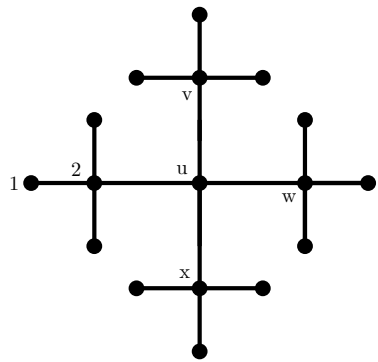
\begin{figure}
\begin{center}
\psscalebox{0.7 0.7}
{
\begin{pspicture}(0,-8.400001)(6.967837,-1.7043265)
\psdots[linecolor=black, dotsize=0.3](0.42,-5.052163)
\psdots[linecolor=black, dotsize=0.3](1.62,-5.052163)
\psdots[linecolor=black, dotsize=0.3](3.62,-5.052163)
\psdots[linecolor=black, dotsize=0.3](5.62,-5.052163)
\psdots[linecolor=black, dotsize=0.3](1.62,-6.252163)
\psdots[linecolor=black, dotsize=0.3](1.62,-3.852163)
\psdots[linecolor=black, dotsize=0.3](3.62,-3.0521631)
\psline[linecolor=black, linewidth=0.08](1.62,-3.852163)(1.62,-6.252163)(1.62,-6.252163)
\psline[linecolor=black, linewidth=0.08](3.62,-3.852163)(3.62,-6.252163)(3.62,-6.252163)
\rput[bl](0.0,-5.192163){1}
\rput[bl](1.18,-4.9121633){2}
\psdots[linecolor=black, dotsize=0.3](6.82,-5.052163)
\psdots[linecolor=black, dotsize=0.3](5.62,-3.852163)
\psdots[linecolor=black, dotsize=0.3](5.62,-6.252163)
\psline[linecolor=black, linewidth=0.08](5.62,-3.852163)(5.62,-6.252163)(5.62,-5.052163)(6.82,-5.052163)(6.82,-5.052163)
\psdots[linecolor=black, dotsize=0.3](4.82,-3.0521631)
\psdots[linecolor=black, dotsize=0.3](2.42,-3.0521631)
\psdots[linecolor=black, dotsize=0.3](3.62,-1.8521631)
\psline[linecolor=black, linewidth=0.08](3.62,-1.8521631)(3.62,-4.252163)(3.62,-4.252163)
\psline[linecolor=black, linewidth=0.08](2.42,-3.0521631)(4.82,-3.0521631)(4.82,-3.0521631)
\psdots[linecolor=black, dotsize=0.3](3.62,-7.052163)
\psdots[linecolor=black, dotsize=0.3](4.82,-7.052163)
\psdots[linecolor=black, dotsize=0.3](2.42,-7.052163)
\psdots[linecolor=black, dotsize=0.3](3.62,-8.252163)
\psline[linecolor=black, linewidth=0.08](0.42,-5.052163)(6.82,-5.052163)(6.82,-5.052163)
\psline[linecolor=black, linewidth=0.08](3.62,-5.052163)(3.62,-8.252163)(3.62,-8.252163)
\psline[linecolor=black, linewidth=0.08](2.42,-7.052163)(4.82,-7.052163)(4.82,-7.052163)
\rput[bl](3.18,-4.8721633){u}
\rput[bl](3.28,-3.4721632){v}
\rput[bl](5.18,-5.4121633){w}
\rput[bl](3.24,-6.8521633){x}
\end{pspicture}
}
\end{center}
\caption{Neopentane related to the proof of Theorem \ref{neop}.}\label{neop1}
\end{figure}

\begin{figure}
	\begin{center}
		\psscalebox{0.7 0.7} 
		{
			\begin{pspicture}(0,-5.21)(13.93,-2.43)
			\psdots[linecolor=black, dotsize=0.3](0.42,-3.81)
			\psdots[linecolor=black, dotsize=0.3](1.62,-3.81)
			\psdots[linecolor=black, dotsize=0.3](2.82,-3.81)
			\psdots[linecolor=black, dotsize=0.3](4.02,-3.81)
			\psdots[linecolor=black, dotsize=0.3](5.22,-3.81)
			\psdots[linecolor=black, dotsize=0.3](1.62,-5.01)
			\psdots[linecolor=black, dotsize=0.3](2.82,-5.01)
			\psdots[linecolor=black, dotsize=0.3](4.02,-5.01)
			\psdots[linecolor=black, dotsize=0.3](1.62,-2.61)
			\psdots[linecolor=black, dotsize=0.3](2.82,-2.61)
			\psdots[linecolor=black, dotsize=0.3](4.02,-2.61)
			\psline[linecolor=black, linewidth=0.08](0.42,-3.81)(5.22,-3.81)(5.22,-3.81)
			\psline[linecolor=black, linewidth=0.08](1.62,-2.61)(1.62,-5.01)(1.62,-5.01)
			\psline[linecolor=black, linewidth=0.08](2.82,-2.61)(2.82,-5.01)(2.82,-5.01)
			\psline[linecolor=black, linewidth=0.08](4.02,-2.61)(4.02,-5.01)(4.02,-5.01)
			\rput[bl](0.0,-3.95){1}
			\rput[bl](1.18,-3.67){2}
			\rput[bl](1.22,-2.75){3}
			\rput[bl](2.46,-3.67){4}
			\rput[bl](2.38,-2.75){5}
			\rput[bl](1.16,-5.15){6}
			\rput[bl](3.66,-3.65){8}
			\rput[bl](2.22,-5.11){12}
			\rput[bl](4.66,-3.65){16}
			\rput[bl](3.46,-5.17){20}
			\rput[bl](3.42,-2.73){15}
			\psdots[linecolor=black, dotsize=0.3](6.42,-3.81)
			\psdots[linecolor=black, dotsize=0.3](5.22,-2.61)
			\psdots[linecolor=black, dotsize=0.3](5.22,-5.01)
			\psline[linecolor=black, linewidth=0.08](5.22,-2.61)(5.22,-5.01)(5.22,-3.81)(6.42,-3.81)(6.42,-3.81)
			\rput[bl](4.62,-5.15){48}
			\rput[bl](5.92,-3.65){32}
			\psdots[linecolor=black, dotsize=0.3](7.62,-3.81)
			\psdots[linecolor=black, dotsize=0.3](6.42,-5.01)
			\psdots[linecolor=black, dotsize=0.3](6.42,-2.61)
			\psdots[linecolor=black, dotsize=0.3](7.62,-2.61)
			\psdots[linecolor=black, dotsize=0.3](7.62,-5.01)
			\psdots[linecolor=black, dotsize=0.1](9.22,-3.81)
			\psdots[linecolor=black, dotsize=0.1](9.62,-3.81)
			\psdots[linecolor=black, dotsize=0.1](10.02,-3.81)
			\psline[linecolor=black, linewidth=0.08](6.42,-3.81)(8.82,-3.81)(7.62,-3.81)
			\psline[linecolor=black, linewidth=0.08](7.62,-2.61)(7.62,-5.01)(7.62,-5.01)
			\psline[linecolor=black, linewidth=0.08](6.42,-2.61)(6.42,-5.01)(6.42,-5.01)
			\psdots[linecolor=black, dotsize=0.3](11.62,-3.81)
			\psdots[linecolor=black, dotsize=0.3](11.62,-2.61)
			\psdots[linecolor=black, dotsize=0.3](11.62,-5.01)
			\psdots[linecolor=black, dotsize=0.3](12.82,-3.81)
			\psline[linecolor=black, linewidth=0.08](10.42,-3.81)(12.82,-3.81)(12.82,-3.81)
			\psline[linecolor=black, linewidth=0.08](11.62,-2.61)(11.62,-5.01)(11.62,-5.01)
			\rput[bl](7.12,-3.67){64}
			\rput[bl](13.16,-3.91){{$2^{n+1}$}}
			\rput[bl](11.74,-3.69){{$2^n$}}
			\rput[bl](5.88,-5.13){96}
			\rput[bl](5.86,-2.73){80}
			\rput[bl](6.9,-5.13){192}
			\rput[bl](6.96,-2.75){160}
			\rput[bl](4.64,-2.73){40}
			\rput[bl](10.3,-2.85){{$5(2^{n-1})$}}
			\rput[bl](10.6,-5.21){{$3(2^n)$}}
			\end{pspicture}
		}
	\end{center}
	\caption{ Alkanes $C_nH_{2n+2}$ for every $n>5$ related to the proof of Theorem \ref{Alkane}}\label{Alk-pic}
\end{figure}
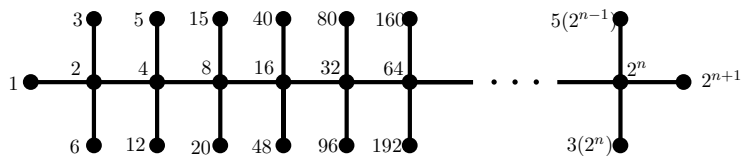

    \begin{theorem} \label{Alkane}
		The Alkanes $C_nH_{2n+2}$ are $G_\phi$-graphs.
	\end{theorem}
	
	\begin{proof}
	For $n=1,2,3$ we have Methane, Ethane and Propane respectively (Figure \ref{chem1}). For $n=4$ we have Butane (Figure \ref{chem2}) and for $n=5$ we have Pentane (Figure \ref{chem3}). So for every $n>5$ it suffices to consider
	$$A=\{ 3,6,5,12,15,20,40,80,\ldots,5(2^{n-1}),48,96,\ldots,3(2^n),2^{n+1}\}.$$ 
	As we see in Figure \ref{Alk-pic}, $G_{\phi}(A)$ is $C_nH_{2n+2}$.
	\qed
	\end{proof}

A nanostructure is an object of middle of the road size among minute and atomic structures. This is approximately due to a physical measurement lesser than 100 nanometers, extending from groups of iotas to dimensional layers. Nanobiotechnology is a hurriedly boosting territory of logical and mechanical open door that applies the apparatuses and procedures of nanofabrication to fabricate gadgets for examining biosystems \cite{Nano}.

From a polymer science perspective, dendrimers are almost immaculate monodisperse macromolecules with a customary and exceptionally fanned three-dimensional design. The nanostar dendrimer is a piece of another gathering of macroparticles that seem, by all accounts, to be photon channels simply like counterfeit reception apparatuses. These macromolecules and all those more absolutely containing phosphorus are utilized in the development of nanotubes, smaller scale macrocapsules, nanolatex, shaded glasses, concoction sensors, and adjusted terminals \cite{Nano}. The graph $D_2$ in Figure \ref{nanostar} is the tree structure of a nanostar has grown 2 stages (see \cite{Saeid,Emeric}).  We state and prove the following result.

    \begin{theorem} \label{nano}
		Nanostar $D_2$ is a $G_\phi$-graph.
	\end{theorem}
	
	\begin{proof}
	It suffices to consider $A=\{3,256,376,384,564\}$. As wee see in the Figure \ref{nanostar}, $D_2$ is a $G_\phi$-graph.
	\qed
	\end{proof}

We think that there is no dendrimer nanostar which has grown at least 4 stages, as a $G_\phi$-graph, but until now all attempts to
prove this failed. So, we end this paper by proposing the following conjecture. 

\begin{conjecture} 
There is no  dendrimer nanostar which has grown at least 4 stages, as a $G_\phi$-graph.
\end{conjecture}

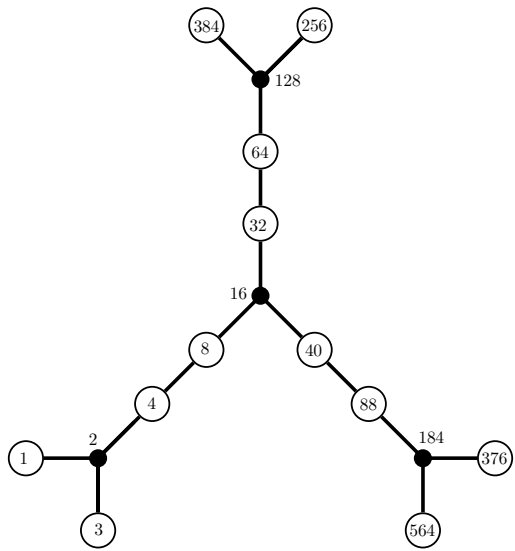
\begin{figure}
\begin{center}
\psscalebox{0.6 0.6} 
{
\begin{pspicture}(0,-6.0000005)(11.205557,6.005556)
\psdots[linecolor=black, dotsize=0.4](5.6027775,-0.39722168)
\psdots[linecolor=black, fillstyle=solid, dotstyle=o, dotsize=0.8, fillcolor=white](5.6027775,1.2027783)
\psdots[linecolor=black, fillstyle=solid, dotstyle=o, dotsize=0.8, fillcolor=white](5.6027775,2.8027782)
\psdots[linecolor=black, dotsize=0.4](5.6027775,4.402778)
\psdots[linecolor=black, dotsize=0.4](9.202778,-3.9972217)
\psdots[linecolor=black, dotsize=0.4](2.0027778,-3.9972217)
\psdots[linecolor=black, fillstyle=solid, dotstyle=o, dotsize=0.8, fillcolor=white](10.802777,-3.9972217)
\psdots[linecolor=black, fillstyle=solid, dotstyle=o, dotsize=0.8, fillcolor=white](9.202778,-5.597222)
\psdots[linecolor=black, fillstyle=solid, dotstyle=o, dotsize=0.8, fillcolor=white](0.4027777,-3.9972217)
\psdots[linecolor=black, fillstyle=solid, dotstyle=o, dotsize=0.8, fillcolor=white](2.0027778,-5.597222)
\psline[linecolor=black, linewidth=0.08](0.8027777,-3.9972217)(2.0027778,-3.9972217)(2.0027778,-5.1972218)(2.0027778,-3.9972217)(2.8027778,-3.1972218)(2.8027778,-3.1972218)
\psline[linecolor=black, linewidth=0.08](9.202778,-3.9972217)(9.202778,-5.1972218)(9.202778,-5.1972218)
\psline[linecolor=black, linewidth=0.08](9.202778,-3.9972217)(10.402778,-3.9972217)(10.402778,-3.9972217)
\psline[linecolor=black, linewidth=0.08](5.6027775,-0.39722168)(5.6027775,0.8027783)(5.6027775,0.8027783)
\psline[linecolor=black, linewidth=0.08](5.6027775,1.6027783)(5.6027775,2.4027784)(5.6027775,2.4027784)
\psline[linecolor=black, linewidth=0.08](5.6027775,4.402778)(5.6027775,3.2027783)(5.6027775,3.2027783)
\rput[bl](0.26277772,-4.117222){1}
\rput[bl](1.8027778,-3.6972218){2}
\rput[bl](4.9227777,-0.4572217){16}
\rput[bl](5.3627777,1.0427784){32}
\rput[bl](5.4027777,2.6627784){64}
\rput[bl](5.9227777,4.2627783){128}
\rput[bl](1.9227777,-5.6972218){3}
\rput[bl](9.1027775,-3.6572218){184}
\rput[bl](10.502778,-4.137222){376}
\rput[bl](8.902778,-5.7172217){564}
\psline[linecolor=black, linewidth=0.08](5.6027775,4.402778)(6.802778,5.6027784)(5.6027775,4.402778)(4.4027777,5.6027784)(4.4027777,5.6027784)
\psdots[linecolor=black, dotstyle=o, dotsize=0.8, fillcolor=white](6.802778,5.6027784)
\psdots[linecolor=black, dotstyle=o, dotsize=0.8, fillcolor=white](4.4027777,5.6027784)
\psline[linecolor=black, linewidth=0.08](2.0027778,-3.9972217)(5.6027775,-0.39722168)(5.6027775,-0.39722168)(9.202778,-3.9972217)(9.202778,-3.9972217)
\rput[bl](4.122778,5.442778){384}
\rput[bl](6.509444,5.4827785){256}
\psdots[linecolor=black, dotstyle=o, dotsize=0.8, fillcolor=white](6.802778,-1.5972217)
\psdots[linecolor=black, dotstyle=o, dotsize=0.8, fillcolor=white](8.002778,-2.7972217)
\psdots[linecolor=black, dotstyle=o, dotsize=0.8, fillcolor=white](4.4027777,-1.5972217)
\psdots[linecolor=black, dotstyle=o, dotsize=0.8, fillcolor=white](3.2027776,-2.7972217)
\rput[bl](4.282778,-1.690555){8}
\rput[bl](3.096111,-2.9038885){4}
\rput[bl](6.589444,-1.730555){40}
\rput[bl](7.802778,-2.930555){88}
\end{pspicture}
}
\end{center}
\caption{ Nanostar $D_2$ related to Theorem \ref{nano}}\label{nanostar}
\end{figure}

\section{Acknowledgements} 

The first author would like to thank the Research Council of Norway (NFR Toppforsk Project Number 274526, Parameterized Complexity for Practical Computing) and Department of Informatics, University of
Bergen for their support. Also he is thankful to Michael Fellows and
Michal Walicki for conversations and appreciation to them and Frances
Rosamond for sharing their pearls of wisdom with him during the course
of this research.

\end{document}